\documentclass[a4paper,12pt]{article}

\usepackage[T1]{fontenc}
\usepackage{lmodern}
\usepackage{microtype}
\usepackage{amsmath}
\usepackage{amssymb}
\usepackage{amsthm}
\usepackage[hidelinks]{hyperref}
\usepackage{xcolor}

\usepackage[margin=2.5cm]{geometry}

\newtheorem{theorem}{Theorem}[section]
\newtheorem{proposition}[theorem]{Proposition}

\newtheorem{conjecture}[theorem]{Conjecture}
\newtheorem{lemma}[theorem]{Lemma}

\numberwithin{equation}{section}

\usepackage{thmtools}
\usepackage{thm-restate}

\title{An asymptotic bound for the Bermond--Thomassen Conjecture}

\author{ J\o rgen Bang-Jensen \thanks{JBJ: School of Mathematics,  Shandong University, Jinan 250100, China; Department of Mathematics and Computer Science, University of Southern Denmark, Odense DK-5230, Denmark. \texttt{jbj@imada.sdu.dk}.} \and
	Guanghui Wang\thanks{GHW: School of Mathematics,  Shandong University, Jinan 250100, China.   \texttt{ghwang@sdu.edu.cn}. The author was supported by National Natural Science Foundation of China (No.12231018).} \and
	Yun Wang\thanks{YW: Data Science Institute, Shandong University, Jinan 250100, China. \texttt{yunwang@sdu.edu.cn}. The author was supported by China Postdoctoral Science Foundation (No. 2025M773101, 2026T190825) and National Natural Science Foundation of China (No. 12501489).}}

\date{}
\hypersetup{pdftitle={An asymptotic bound for the Bermond--Thomassen Conjecture}}
\begin{document}

\maketitle

\begin{abstract}
The Bermond--Thomassen Conjecture asserts that every digraph with
minimum out-degree at least $2k-1$ contains $k$ disjoint
directed cycles; it remains open in general. We asymptotically
resolve it: there exist an absolute constant $K$ and a function
$g(k)=o(k)$, independent of $n$, such that every $n$-vertex finite
simple loopless digraph with minimum out-degree at least $2k+g(k)$
contains $k$ disjoint directed cycles for all integers $k\ge K$ and
$n\ge1$. In fact we obtain the explicit
error term $g(k)=O(k^{3/4}\sqrt{\log k})$. 
\end{abstract}

\noindent{\bf Keywords:} \texttt{Bermond--Thomassen Conjecture, fractional cycle packings, disjoint  cycles}

\section{Introduction}
Throughout, digraphs are finite and simple; pairs of
opposite arcs are allowed, and every cycle is directed and has length at
least two. We write $\delta^+(D)$ for the minimum out-degree of a
digraph $D$ and generally follow the notation and terminology of \cite{bang2009}.

Finding disjoint cycles under degree conditions is a classical problem
in extremal graph theory. In the undirected setting, the
Corr\'adi--Hajnal Theorem~\cite{corradiAMASH14} asserts that every graph with at
least $3k$ vertices and minimum degree at least $2k$ contains $k$
disjoint cycles. The directed analogue is considerably more delicate.
The transitive tournament has 
maximum possible total degree at every vertex, yet contains no directed
cycle; thus no condition on the total degree alone can force a directed
cycle, and one must impose a lower bound on the out-degree. This is the
natural directed counterpart of the undirected condition: every finite
digraph with positive minimum out-degree contains a directed cycle. The
central question is how large the minimum out-degree must be to
guarantee a prescribed number of disjoint directed cycles.

For a positive integer $k$, let $f(k)$ be the least integer $d$ such
that every digraph with minimum out-degree at least $d$ contains $k$
disjoint directed cycles. In 1981, Bermond and
Thomassen proposed the following conjecture.

\begin{conjecture}[Bermond--Thomassen Conjecture~\cite{bermondJGT5}]
For every positive integer $k$, every digraph with minimum out-degree
at least $2k-1$ contains $k$ disjoint directed cycles.
\end{conjecture}

The threshold $2k-1$ is best possible. Indeed, the complete symmetric
digraph on $2k-1$ vertices has minimum out-degree $2k-2$, while it contains at most
$k-1$ disjoint directed cycles; hence $f(k)\ge2k-1$.

Thomassen~\cite{thomassenC3} proved in 1983 that $f(k)$ is finite for every
$k$, obtaining $f(k)\le(k+1)!$, and settled the conjecture for $k=2$.
Alon~\cite{alonJCTB68} then introduced a probabilistic approach and showed
that $f(k)\le64k$, establishing that the required out-degree grows
linearly in the number of cycles. Buci\'c~\cite{bucicDM341} refined this
approach--incorporating the known case of three disjoint cycles--and
improved the bound to $f(k)\le18k$. While these results determine the
correct order of growth, they leave a wide gap to the conjectured
threshold.

Progress on the exact conjecture has followed a complementary path. The
case $k=1$ is immediate; after Thomassen's elegant proof for the case $k=2$, the case $k=3$
was proved by Lichiardopol, P\'or, and Sereni~\cite{lichiardopolSJDM23}. Among
tournaments, Bang-Jensen, Bessy, and Thomass\'e~\cite{bangjensenJGT75} showed
that minimum out-degree at least $2k-1$ guarantees $k$ disjoint
directed triangles. More recently, Bessy, Muis, Sereni, Steiner, and
Wiederrecht~\cite{bessyARX26} proved that the conjectured degree condition
guarantees $k$ disjoint near-directed cycles, that is, oriented cycles that can be made directed by flipping the direction of at most one arc. The exact conjecture for
general digraphs remains open for $k\ge4$.

\medskip
In this paper we determine the asymptotic value of $f(k)$, resolving
the Bermond--Thomassen conjecture to leading order with an explicit
sublinear error term.% that is uniform\jbjcomment{I think 'independent of' is a better word here }in the order of the host digraph.

\begin{restatable}{theorem}{MainThm}\label{thm:main}
    There exist absolute constants $C>0$ and $K$ such that, for all
integers $k\ge K$ and $n\ge1$, every $n$-vertex digraph $D$ with
$$
\delta^+(D)\ge 2k+Ck^{3/4}\sqrt{\log k}
$$
contains $k$ disjoint directed cycles.
\end{restatable}
\iffalse
\begin{theorem}\label{thm:main}
There exist absolute constants $C>0$ and $K$ such that, for all
integers $k\ge K$ and $n\ge1$, every $n$-vertex digraph $D$ with
$$
\delta^+(D)\ge 2k+Ck^{3/4}\sqrt{\log k}
$$
contains $k$ disjoint directed cycles.
\end{theorem}
\fi

In fact, our method reaches the explicit choice
$K=10^{12}$ and $C=64\sqrt2$, see Section~\ref{sec:mainproof}.   Together with the complete symmetric digraph construction,
$ 2k-1\le f(k)\le 2k+O\!\left(k^{3/4}\sqrt{\log k}\right), $
so $\lim_{k\to\infty}f(k)/k=2$. Equivalently, for every $\varepsilon>0$, minimum out-degree at least $(2+\varepsilon)k$ guarantees $k$ disjoint directed cycles once $k$ is sufficiently large. Moreover, by Lemma~\ref{lem:large-order-equivalence}, the condition $k\ge K$ cannot be replaced by an assumption of sufficiently large order $n$ if we wish to retain the same minimum out-degree bound.  Lemma~\ref{lem:large-order-equivalence} is formally stated and proved in Section~\ref{sec:preliminaries}.

\bigskip
\textbf{The rest of the paper is organized as follows.}
We begin with Section~\ref{sec:notation}, which contains some notation and terminology.
In Section~\ref{sec:preliminaries} we recall fractional cycle packings and state the standard probabilistic estimates used later. 
Section~\ref{sec:order} bounds the order of a minimum counterexample
by means of contractions and an independent-set argument.
In Section~\ref{sec:iteration} we prove a lower bound on the fractional
cycle packing number (Theorem~\ref{thm:fractional}), construct a cycle
distribution with bounded vertex incidence
(Lemma~\ref{lem:distribution}), and show that cycles can be deleted
successively so that a suitable potential function does not increase
(Lemma~\ref{lem:deletion}). Finally, in Section~\ref{sec:mainproof} we
combine these ingredients to prove the main theorem.

\section{Notation}\label{sec:notation}

For a positive integer $a$, let $[a]:=\{1,2,\ldots,a\}$; for integers $a<b$, let $[a,b]:=\{a,a+1,\ldots,b\}$. We write $x\ll y$ when $x$ is sufficiently small as a function of $y$. The notation $t=a\pm b$ means $a-b\leq t\leq a+b$. All logarithms are natural.

Let $D=(V,E)$ be a digraph.  The order of $D$ is the number of vertices in $D$,  denoted by $|V(D)|$ or $|D|$.  We will denote an edge oriented from $u$ to $v$ by $uv$. We use $D[X]$ to denote the subdigraph induced by a vertex set $X$. Let $D-X = D[V\backslash X]$.  For $v\in V(D)$, we denote the set of out-neighbors and in-neighbors of $v$ in $D$ by $N^+(v)$ and $N^-(v)$ respectively, and we write $d^+(v)=|N^+(v)|$ and $d^-(v)=|N^-(v)|$ for the out-degree and in-degree of $v$, respectively. For $S\subseteq V(D)$, we write $N^-(v,S)=N^-(v)\cap S$, $N^+(v,S)=N^+(v)\cap S$. Set $d^-(v,S)=|N^-(v,S)|$ and $d^+(v,S)=|N^+(v, S)|$.   %For two subsets $A,B$ of $V(G)$, we use $E(A,B)$ to denote the set of edges from $A$ to $B$ and write $e(A,B) = |E(A,B)|$.  
Write $\nu(D)$ for the maximum number of disjoint cycles
in $D$. A digraph $D=(V,E)$ is strongly connected (or just strong) if it has a directed path from $x$ to $y$ for every ordered pair $x,y$ of vertices.

\section{Preliminaries}\label{sec:preliminaries}

\subsection{Fractional cycle packings}\label{subsec:fractional}

Fractional cycle packings with vertex capacity one were studied by
Seymour~\cite{seymourC15}, who related their value to the size of a minimum feedback
vertex set. We require a complementary estimate: a lower bound on the
fractional packing number in terms of the minimum out-degree, which we
prove in this section.
Let $\mathcal C(G)$ be the set of cycles of $G$.
A \emph{fractional cycle packing} assigns a number $x_C\ge0$ to
each $C\in\mathcal C(G)$ so that
$\sum_{C\ni v}x_C\le1$ for every vertex $v$. Its total weight is
$\sum_C x_C$, and the largest possible total weight is denoted by
$\nu^\ast(G)$. If $G$ has no cycle, we set $\nu^\ast(G)=0$.

The dual object is a \emph{fractional feedback vertex set}: it assigns
weights $y_v\ge0$ to vertices so that
$\sum_{v\in V(C)}y_v\ge1$ for every cycle $C$.
Write $y(X)=\sum_{v\in X}y_v$ for a set of vertices $X$.
Both problems are feasible: the primal by setting $x_C=0$ for every
cycle, and the dual by setting $y_v=1$ for every vertex. By finite linear programming duality~\cite[Chapter~7]{schrijverTLIP1986},
\begin{equation}\label{eq:duality}
\nu^\ast(G)=\min\{y(V(G)):\,y\text{ is a fractional feedback vertex set}\}.
\end{equation} 

Both optima are attained: the primal feasible region is compact since
$0\le x_C\le 1$, and in the dual any weight $y_v>1$ can be replaced by
$1$ without violating a cycle constraint. We may therefore choose an
optimal dual assignment with $0\le y_v\le 1$.

\subsection{Standard estimates}\label{subsec:estimate}

We will use the following standard estimates. %\jbjcomment{Perhaps we should put a reference for these estimates}

\begin{lemma}[Chernoff bound]\label{lem:chernoff}
Let $X\sim\mathrm{Bin}(n,p)$ with mean $\mu=np$. Then for $0<\tau<1$,
$$
\mathbb P\bigl[|X-\mu|>\tau\mu\bigr]\le2e^{-\tau^2\mu/3}.
$$
\end{lemma}

\begin{lemma}[Markov's inequality]\label{lem:markov}
If $X\ge0$ is a random variable, then for every $t>0$,
$$
\mathbb P[X\ge t]\le\frac{\mathbb E[X]}{t}.
$$
\end{lemma}

% \begin{lemma}[Azuma's Inequality  {\cite[Chapter 7]{alonspencerbook}}]
% \label{lem:azuma_hoeffding}
% Let $Z_0,Z_1,\dots,Z_n$ be a martingale, and suppose that for each $i\in [n]$, $|Z_i-Z_{i-1}|\le c_i$ almost surely. Then for any $\lambda>0$,
% \[
% \mathbb P\left[Z_n-Z_0\ge \lambda\right] \le \exp\left(-\frac{\lambda^2}{2\sum_{i=1}^n c_i^2}\right).
% \]
% Since $-Z_0,-Z_1,\dots,-Z_n$ is also a martingale,
% \[
% \mathbb P\left[Z_n-Z_0\le -\lambda\right] \le \exp\left(-\frac{\lambda^2}{2\sum_{i=1}^n c_i^2}\right),
% \]
% and consequently
% \[
% \mathbb P\left[\left|Z_n-Z_0\right|\ge \lambda\right] \le 2\exp\left(-\frac{\lambda^2}{2\sum_{i=1}^n c_i^2}\right).
% \]
% \end{lemma}

\begin{lemma}[Bernoulli's inequality]\label{lem:bernoulli}
For every integer $q\ge1$ and every real $x\le1$,
$$
(1-x)^q\ge1-qx.
$$
In particular, if $0\le x<1/(q+1)$, then
$(1-x)^{q+1}\ge1-(q+1)x>0$.
\end{lemma}

\begin{lemma}\label{lem:mvt}
For $q\ge1$ and $0\le x<1/(q+1)$,
$$
(1-x)^{-q}\le1+\frac{qx}{1-(q+1)x}.
$$
\end{lemma}
\begin{proof}
For $0\le x<1/(q+1)$, apply the mean value theorem to $f(t)=(1-t)^{-q}$ on $[0,x]$. There exists $c\in(0,x)$ such that
$
(1-x)^{-q}-1 = qx(1-c)^{-q-1}.
$
Since $0<c<x$, we have $(1-c)^{-q-1}\le (1-x)^{-q-1}$, so
$
(1-x)^{-q}-1\le qx(1-x)^{-q-1}.
$
Bernoulli's inequality gives $(1-x)^{q+1}\ge1-(q+1)x>0$, and therefore
$(1-x)^{-q-1}\le1/(1-(q+1)x)$. Combining these two estimates, the claim follows.
\end{proof}

\subsection{Equivalence properties}\label{subsec:equivalence}

In this subsection, we establish an equivalence principle that allows us to reduce general cycle packing problems in digraphs to the setting of sufficiently large strong out-regular digraphs.

\begin{lemma} \label{lem:large-order-equivalence}
Fix integers $k\ge1$ and $d\ge1$. The following statements are equivalent.
\begin{enumerate}
    \item[\textup{(i)}] Every finite simple loopless digraph $D$ satisfying $\delta^+(D)\ge d $ contains $k$ disjoint directed cycles.
    \item[\textup{(ii)}] There exists an integer $n_0=n_0(k,d)$ such that every finite simple loopless digraph $D$ satisfying $|V(D)|\ge n_0 $ and $\delta^+(D)\ge d $ contains $k$ disjoint directed cycles.
    \item[\textup{(iii)}] There exists an integer $n_0=n_0(k,d)$ such that every strong $d$-out-regular digraph $D$ satisfying $|V(D)|\ge n_0 $ contains $k$ disjoint directed cycles.
\end{enumerate}
The same equivalence holds if all digraphs above are required to be oriented graphs, that is 2-cycles are not allowed.
\end{lemma}

\begin{proof}
The implications $\textup{(i)}\Longrightarrow\textup{(ii)}\Longrightarrow\textup{(iii)}$ follow directly from definitions. We prove that $\textup{(iii)}\Longrightarrow\textup{(i)}$. Assume statement (i) does not hold. Then there exists a finite simple loopless digraph $D$ such that $\delta^+(D)\ge d$ and $\nu(D)<k$.

First retain exactly $d$ outgoing arcs at every vertex of $D$, deleting all other arcs, and denote the resulting spanning subdigraph by $D'$. Then $d_{D'}^+(v)=d$ for all $v\in V(D')$, while deleting arcs cannot create new directed cycles, and hence $\nu(D')\le \nu(D)<k$. The strong components of $D'$ admit an acyclic ordering. Let $H$ be a terminal strong component in this ordering.
  Since no arc of $D'$ leaves $H$, every vertex of $H$ has all its $d$ outgoing arcs inside $H$. Consequently $H$ is strong and $d$-out-regular, and $\nu(H)\le \nu(D')<k$. As $d\ge1$, the digraph $H$ contains an arc.

We now describe an order increasing operation. Let $G$ be any finite simple loopless digraph with an arc $uv$, and let $m\ge1$. Put $A=N_G^+(u)\setminus\{v\}$. Introduce new vertices $x_1,\ldots,x_m$, delete $uv$, keep every other old arc, and add $ux_1,\ x_1x_2,\ldots,\ x_{m-1}x_m,\ x_mv$, together with $x_iw$ for $i\in[m],\ w\in A$. Denote the resulting digraph by $G_m$. Every old vertex keeps its out-degree. Moreover, $d_{G_m}^+(x_i)=d_G^+(u)$ for $i\in[m]$, and therefore $\delta^+(G_m)=\delta^+(G)$.

We claim that $\nu(G_m)=\nu(G)$. Indeed, every directed cycle of $G_m$ meeting a new vertex must contain $u$. Such a cycle contains a consecutive segment $ux_1x_2\cdots x_iw$ for some $w\in N_G^+(u)$. Replacing this segment by the old arc $uw$ gives a directed cycle of $G$ with exactly the same old vertices. Since every cycle meeting the new vertices contains $u$, at most one member of a disjoint cycle family can meet the new vertices. Thus an entire packing in $G_m$ can be compressed simultaneously to a packing of the same cardinality in $G$, proving $\nu(G_m)\le \nu(G)$. Conversely, take a maximum packing in $G$. At most one of its members uses the arc $uv$. If such a member exists, replace its occurrence of $uv$ by $ux_1x_2\cdots x_mv$. All other members remain unchanged, and the inserted vertices are new, so disjointness is preserved. Hence $\nu(G_m)\ge \nu(G)$. 

The same construction preserves strong connectivity: every old directed path can be lifted by replacing $uv$, when necessary, by the new path, while every new vertex is reachable from $u$ and can reach $v$. By construction, the new digraph $G_m$ is an oriented graph whenever $G$ is an oriented graph. Applying the operation above to an arc of $H$, we obtain, for every integer $N\ge |V(H)|$, a strong $d$-out-regular digraph $H_N$ of order exactly $N$ with $\nu(H_N)=\nu(H)<k$. This produces counterexamples to statement (iii) with arbitrarily large order, contradicting statement (iii). Hence $\textup{(iii)}\Rightarrow\textup{(i)}$. This proves the equivalence of (i)--(iii).
\end{proof}

As an additional consequence of the construction, a single finite counterexample automatically produces counterexamples of every larger order. %suppose that (i) fails and $D$ is any counterexample. For every integer $N\ge |V(D)|$, applying the order increasing operation directly to an arc of $D$ yields a counterexample $D_N$ of order $N$ satisfying $\delta^+(D_N)\ge d$ and $\nu(D_N)=\nu(D)<k$. In other words, 

\section{The order of a minimum counterexample}\label{sec:order}

Fix integers $k\ge2$ and $r\ge2k$. In this section, a \textit{counterexample} is
a digraph $D$ with $\delta^+(D)\ge r$ and $\nu(D)<k$. The parameters
$r$ and $k$ remain fixed whenever minimum order is used.

The contraction argument in Lemma~\ref{lem:structure} and the
independent set count in Proposition~\ref{prop:order} follow
Alon~\cite[Claims~1 and~2]{alonJCTB68}; see also
Buci\'c~\cite[Propositions~3.1 and~3.2]{bucicDM341}.
We give the arguments for the fixed pair $(r,k)$ used here.

Lemma~\ref{lem:loops} justifies contractions that create loops,
without an assumption on $f(k-1)$.
Contraction may create a loop. The next lemma allows us to use such a
contraction without changing the minimum possible order of a
counterexample. For a digraph $G$ in which loops are allowed, count a
loop once in the out-degree and regard it as a cycle of length one.
Write $\nu_L(G)$ for the maximum number of disjoint cycles under
this convention.

\begin{lemma}\label{lem:loops}
If a loopless counterexample exists, its minimum possible order equals
the minimum possible order of a counterexample in which loops are
allowed.
\end{lemma}

\begin{proof}Recall that a counterexample is a digraph $D$ with $\delta^+(D)\ge r$ and $\nu(D)<k$. 
Write $m_0$ for the minimum order of a loopless counterexample and
$m_L$ for the minimum order of a counterexample in which loops are
allowed. Thus it suffices to show $m_0=m_L$.
The inequality $m_L\le m_0$ is immediate, since every loopless
digraph is also a digraph in which loops are allowed.

It remains to prove $m_0\le m_L$. Let $L$ be the set of looped
vertices of a digraph $G$. Then $\nu_L(G)=|L|+\nu(G-L)$.  %Indeed, at most $|L|$ cycles in a disjoint family meet $L$, and all remaining cycles lie in $G-L$, giving $\nu_L(G)\le |L|+\nu(G-L)$; the reverse inequality follows by taking the loops on $L$ together with a maximum family in $G-L$. 
Choose a counterexample of minimum order $m_L$ among digraphs in
which loops are allowed, and, subject to this, with as many arcs as
possible.  We shall use this counterexample to build a loopless counterexample
of the same order, which will give $m_0\le m_L$.

We claim that every looped vertex is adjacent in both
directions to every other vertex. Otherwise, if some $v\in L$ has no
arc to $w$ with $w\ne v$, adding that arc would leave
$\nu_L(G)=|L|+\nu(G-L)$ unchanged, since neither $L$ nor
$G-L$ is affected, while preserving the out-degree condition,
contradicting the choice of $G$.

Since every out-degree is at most $m_L$, we have $m_L\ge r$. Equality
$m_L=r$ would force every possible arc and every loop to be present,
giving $\nu_L(G)=r\ge k$, contrary to $G$ being a counterexample.
Thus $m_L\ge r+1$. Now delete all loops of $G$, retaining every vertex and every arc
between distinct vertices. A formerly looped vertex has out-degree
$m_L-1\ge r$, while every other out-degree is unchanged. Deleting arcs
cannot create a family of $k$ disjoint cycles. The resulting digraph
is therefore a loopless counterexample of order $m_L$, so
$m_0\le m_L$ and this proves the claim.
\end{proof}

The contraction argument below follows Thomassen~\cite{thomassenC3} (also can be found in Alon~\cite[Claims~1 and~2]{alonJCTB68}); we restate it for the fixed pair $(r,k)$ used here. %\jbjcomment{Thomassen already made these observations in his 1983 paper so we should refer to his paper instead, I think}

\begin{lemma}\label{lem:structure}
If a counterexample exists, then there exists a counterexample of minimum
order that is $r$-out-regular. In any such digraph $D$, the endvertices
of every arc have a common in-neighbour, and $D[N^-(v)]$ contains a cycle for every vertex $v$.
\end{lemma}

\begin{proof}
Choose a counterexample of minimum order, and retain exactly $r$
outgoing arcs at each vertex. The resulting spanning subdigraph $D$
is still a counterexample of the same order. By
Lemma~\ref{lem:loops}, $D$ is also of minimum order among
counterexamples in which loops are allowed.

We first show that the endvertices of every arc have a common
in-neighbour. Suppose, for a contradiction, that an arc $uv$ has no
common in-neighbour. Delete all outgoing arcs of $u$ except $uv$, and
identify $u$ and $v$ to a single vertex $z$. The outgoing arcs of $z$
are those inherited from $v$; we retain a loop at $z$ precisely when
$vu$ is also an arc of $D$, and discard the loop arising solely from $uv$.
Parallel arcs are identified. Since $v$ has $r$ distinct out-neighbours,
$z$ has $r$ distinct outgoing arcs, counting a possible loop. By the
assumption that $uv$ has no common in-neighbour, no vertex outside
$\{u,v\}$ loses an out-neighbour. Hence the resulting digraph $D'$ has
minimum out-degree at least $r$.

We claim that every family of disjoint cycles in $D'$ lifts to a
family of the same size in $D$. Cycles avoiding $z$ are unchanged. A
loop at $z$ is lifted to a 2-cycle $uvu$. For a longer cycle through
$z$, replace $z$ by the directed path $uv$ if the entering arc originally
terminated at $u$, and by $v$ alone if it terminated at $v$. Since a
disjoint family contains at most one cycle through $z$, these
replacements preserve disjointness. Thus $\nu_L(D')\le\nu(D)<k$,
so $D'$ is a counterexample of order smaller than $|D|$, contradicting
the choice of $D$. Hence the endvertices of every arc have a common
in-neighbour.

Finally, every vertex $v$ has an in-neighbour, since  otherwise deleting $v$
would leave all remaining out-degrees unchanged and yield a smaller
counterexample. For $u\in N_D^-(v)$, a common in-neighbour of  $u$ and $v$ lies in $N_D^-(v)$ and points to $u$. Therefore the
non-empty digraph $D[N_D^-(v)]$ has minimum in-degree at least one, and
following incoming arcs until a vertex repeats produces a cycle.
\end{proof}

\begin{proposition}\label{prop:order}
If a counterexample exists, one of minimum order has at most
$(k-1)(r^2-r+1)$ vertices. In particular, when $r=3k$, its order is
less than $9k^3$.
\end{proposition}

\begin{proof}
Choose $D$ as in Lemma~\ref{lem:structure}, and put $n=|V(D)|$.
Construct an undirected graph $J$ on $V(D)$ by joining distinct vertices
$u$ and $v$ when $N^-(u)\cap N^-(v)\ne\varnothing$.
An independent set in $J$ corresponds to pairwise disjoint
in-neighbourhoods in $D$. Each of these in-neighbourhoods contains
a cycle, so the independence number of $J$ satisfies
$\alpha(J)\le\nu(D)\le k-1$.

Every vertex $w\in V(D)$ has out-degree $r$, so $w$ lies in exactly $r$
of the sets $N_D^-(v)$. For any two of these in-neighbourhoods, their
intersection contains $w$, and hence the corresponding two vertices of
$D$ are adjacent in $J$. Thus $w$ accounts for at most $\binom r2$
edges of $J$. Summing over all $n$ vertices of $D$, the total number of
incidences between vertices of $D$ and edges of $J$ is at most
$n\binom r2$. Every edge of $J$ is counted at least once in this sum, so
$e(J)\le n\binom r2$, equivalently $2e(J)\le nr(r-1)$.

Observe that the complement
$\bar{J}$ of $J$ contains no complete graph on $\alpha(J)+1$ vertices, so
Tur\'an's theorem gives $e(\bar{J})\le(1-1/\alpha(J))n^2/2$.
Since $e(J)+e(\bar{J})=\binom n2$, it follows that
$\alpha(J)\ge n^2/(n+2e(J))$. Hence
$$
 k-1\ge\alpha(J)\ge\frac{n^2}{n+2e(J)}\ge\frac{n}{r^2-r+1}.
$$
This proves the claimed bound, and $n<9k^3$ when $r=3k$.
\end{proof}

\section{Successive deletion of cycles}\label{sec:iteration}

We first prove $\nu^\ast(G)\ge\delta^+(G)/2$. The key step is to find a non-empty
proper part of the digraph whose boundary  %\jbjcomment{is it better to use 'closed out-neighborhood' instead of 'boundary'?} 
is small compared with the
weight outside it. The weighted ball-growing argument we use is a
directed, exact analogue of the region-growing technique of Garg,
Vazirani and Yannakakis~\cite{gargSJC25}; the directed version follows the
spreading metric framework of Even, Naor, Rao and Schieber~\cite{evenJACM47}.

%\jbjcomment{perhaps better to say that $G$ is an oriented graph. Otherwise $y$ is not a fractional feedback vertex set} 

\begin{lemma}\label{lem:boundary}
Let $G$ be a strong oriented graph   
with at least two vertices, and let $y$ be a
fractional feedback vertex set with $0<y_v<1/2$ for every vertex.
There are disjoint sets $U,S\subseteq V(G)$ such that $U$ is non-empty
and proper, every out-neighbour of a vertex of $U$ lies in $U\cup S$,
and $|S|\le 2y(V(G)\setminus U)< |V(G)\setminus U|$.
\end{lemma}

\begin{proof}
Give each arc $vw$ length $y_v$. For a root $a$, let $d(a,v)$ be the
shortest directed distance from $a$ to $v$, with $d(a,a)=0$; these
distances are finite because $G$ is strong. Every nontrivial closed
directed walk has length at least one, since it contains a cycle whose
weight is at least one.

For each root $a$, set $B_a=\{v:d(a,v)<1/2\}$. We claim that
$v\in B_a\setminus\{a\}$ implies $a\notin B_v$. Indeed, otherwise two
paths of length $<1/2$ would concatenate to a closed directed walk of
length $<1$, contradicting the fact that  $y$ is a fractional feedback vertex set.

Switching the order of summation and using the mutual exclusion,
\begin{align*}
\sum_a y_a\sum_{v\in B_a\setminus\{a\}}y_v(1-2y_v)
&=\sum_v y_v(1-2y_v)\sum_{\substack{a:\,v\in B_a\\a\ne v}}y_a\\
&\le\sum_v y_v(1-2y_v)\,y(V(G)\setminus B_v)\\
&=\sum_a y_a(1-2y_a)\,y(V(G)\setminus B_a).
\end{align*}
Since $y_a>0$ for every $a$, some root $a$ satisfies
\begin{equation}\label{eq:root-estimate}
\sum_{v\in B_a\setminus\{a\}}y_v(1-2y_v)
\le(1-2y_a)\,y(V(G)\setminus B_a).
\end{equation}

Fix this root $a$. Since $G$ is strong and has at least two vertices, the root $a$ has an
out-neighbour $w$. The arc $aw$ has length $y_a<1/2$, so
$d(a,w)\le y_a<1/2$, and hence $w\in B_a\setminus\{a\}$. Therefore the
left-hand side in (\ref{eq:root-estimate}) is positive, and the chosen inequality implies
$y(B_a)<y(V(G))$. %\jbjcomment{Please write this calculation.}

For $y_a<t<1/2$, define $B(t)=\{v:d(a,v)\le t\}$, $U(t)=\{v:d(a,v)+y_v\le t\}$ and $S(t)=\{v:d(a,v)\le t<d(a,v)+y_v\}$.  Clearly, $B(t)=U(t)\cup S(t)$. 
For the root $a$, $d(a,a)=0\le t$, so $a\in B(t)$; but
$t>y_a=d(a,a)+y_a$, so $a\notin S(t)$. Hence $a\in U(t)$.
 If $v\in U(t)$ and $vw$ is an arc, then
$d(a,w)\le d(a,v)+y_v\le t$, so every out-neighbour of $U(t)$ lies in
$B(t)$.

Set $g(t)=|S(t)|-2y(S(t))$. We compute the average of $g(t)$ over $t\in(y_a,1/2)$. As noted, the
root $a$ never belongs to $S(t)$. For any other vertex
$v\in B_a\setminus\{a\}$, we have $d(a,v)<1/2$, and $v\in S(t)$ exactly
when
$$
d(a,v)\le t<d(a,v)+y_v.
$$
This interval has length $y_v$. Throughout it, $v$ is counted once in
$|S(t)|$ and its weight $y_v$ is counted once in $y(S(t))$, so its
instantaneous contribution to $g(t)=|S(t)|-2y(S(t))$ is $1-2y_v$.
Over the whole interval, $v$ therefore contributes
$y_v(1-2y_v)$ to the total variation of $g(t)$. Vertices outside $B_a$
never belong to $S(t)$, since their distance from $a$ is at least $1/2$.
By (\ref{eq:root-estimate}), the average of $g(t)$ over $t\in(y_a,1/2)$ is
$$
\frac{1}{1/2-y_a}\sum_{v\in B_a\setminus\{a\}}y_v(1-2y_v)
\le 2\,y(V(G)\setminus B_a).
$$

Since $g$ is piecewise constant, choose $t$ with
$g(t)\le 2\,y(V(G)\setminus B_a)$. As $B(t)\subseteq B_a$, we have
$y(V(G)\setminus B(t))\ge y(V(G)\setminus B_a)$, and therefore
$$
|S(t)|-2y(S(t))=g(t)\le 2y(V(G)\setminus B(t)).
$$
Since $V(G)\setminus U(t)=S(t)\cup(V(G)\setminus B(t))$, this gives $
|S(t)|\le 2y(V(G)\setminus U(t))$. Take $U=U(t)$ and $S=S(t)$. Then $a\in U$, so $U$ is non-empty, and
$y(U)\le y(B_a)<y(V(G))$, so $U$ is proper. All outgoing arcs of $U$
end in $U\cup S$.
\end{proof}

\begin{theorem}\label{thm:fractional}
Every digraph $G$ satisfies $\nu^\ast(G)\ge\delta^+(G)/2$.
\end{theorem}

\begin{proof}
Let $\delta=\delta^+(G)$. Suppose that there exists a digraph $G$ such that $\nu^\ast(G)<\delta/2$ and choose a counterexample $G$ of minimum order. Then $\delta\ge 1$ and thus $G$ has a cycle. 
 By (\ref{eq:duality}), we may choose a
fractional feedback vertex set $y$ with 
$y(V(G))=\nu^\ast(G)<\delta/2$. Since $y$ gives weight at least one to every cycle, $y(V(G))\ge1$, and
hence $\delta\ge3$.

The digraph $G$ is strong. Indeed, a terminal strong component $H$ of $G$ has no out-going 
arcs to the rest of $G$ and hence has  minimum out-degree at least $\delta$. Thus the minimality of $G$ and the fact that $y(V(H))\leq y(V(G))$ implies that $G=H$.
%The restricted weights would have total at most $y(V(G))$, giving a
%smaller counterexample.

First suppose that $y_v\ge1/2$ for some vertex $v$.
Then $G-v$ has minimum out-degree at least $\delta-1$, and the
restricted weights have total $y(V(G))-y_v<(\delta-1)/2$. Thus $G-v$ is again a smaller counterexample, contradiction. Thus we must have $y_v<1/2$ for every
vertex $v\in V(G)$. In particular, $G$ has no 2-cycle,
since the weights on such a cycle would sum to less than one, contradicting that $y$ is a fractional vertex cover.

Next suppose that $y_v=0$ for some vertex $v$. We delete $v$ and, for
every pair of arcs $uv$ and $vw$ in $G$, add the arc $uw$; if several
pairs give the same arc, we keep only one copy. No loop is created in
this process, because $G$ has no 2-cycle. The resulting digraph still has minimum out-degree at least $\delta$. The
restricted weights remain a fractional feedback vertex set: replacing each
new arc $uw$ on a cycle by the two-arc path $u\to v\to w$ gives a closed
directed walk of the same weight in $G$, since $y_v=0$.
Thus the resulting smaller digraph is therefore a counterexample, a
contradiction.

We now have $0<y_v<1/2$ for every vertex, so Lemma~\ref{lem:boundary}
gives disjoint non-empty proper sets $U,S\subseteq V(G)$ such that every
out-neighbour of a vertex of $U$ lies in $U\cup S$ and
$|S|\le2y(V(G)\setminus U)$.
Since $U$ is a proper subset of $V(G)$, the minimality of $G$ gives
$\nu^\ast(G[U])\ge\delta^+(G[U])/2$. The restriction of $y$ to $U$ is a
feasible fractional feedback vertex set on $G[U]$, so
$\nu^\ast(G[U])\le y(U)$. Moreover, every out-neighbour of $U$ lies in
$U\cup S$, so deleting $S$ removes at most $|S|$ out-neighbours at any
vertex of $U$, giving $\delta^+(G[U])\ge\delta-|S|$. Combining these
estimates,
$$
\begin{aligned}
\delta
\le\delta^+(G[U])+|S|&\le2\nu^\ast(G[U])+|S|\\
&\le2y(U)+|S|\\
&\le2y(U)+2y(V(G)\setminus U)\\
&=2y(V(G))<\delta,
\end{aligned}
$$
a contradiction. This completes the proof.
\end{proof}

Let $H$ be a digraph. If all vertices of a cycle 
$C$ are removed from $H$, a remaining vertex $v$ loses exactly 
$|N_H^+(v)\cap V(C)|$ out-neighbours. We seek a probability distribution on the cycles of $H$ 
that bounds both the expectation and the largest possible value of this 
loss, for every vertex.
Fix an integer $n\ge|V(H)|\ge2$, and put $\ell=\log(2n)$.
%The same order bound $n$ will be used for all induced subdigraphs in a sequence of deletions.

Under the minimum out-degree conditions considered in this work, Lemmas~\ref{lem:distribution}-\ref{lem:sequence} allow us to extract an integer disjoint cycle packing whose size is at least $\nu^{\ast}(G)-o(k)$ from a fractional cycle packing of weight $\nu^{\ast}(G)$.

\begin{lemma}\label{lem:distribution}
Suppose that $0<\tau\le1/4$ and
$24\tau^{-2}\ell\le s\le\delta^+(H)$.
There is a probability distribution on the cycles of $H$ such that,
writing $C^\ast$ for a cycle drawn from this distribution, every
vertex $v\in V(H)$ satisfies
$$
\mathbb P[v\in V(C^\ast)]\le\frac{2}{(1-\tau)s}
\quad\text{and}\quad
\mathbb E\bigl[|N_H^+(v)\cap V(C^\ast)|\bigr]
\le\frac{2d_H^+(v)}{(1-\tau)s}.
$$
Moreover, with positive probability  every cycle $C^\ast$ satisfies
$$
|N_H^+(v)\cap V(C^\ast)|
\le\frac{24(1+\tau)\ell\,d_H^+(v)}{\tau^2s}
$$
for every vertex $v$.
\end{lemma}

\begin{proof}
Put $t=\lfloor\tau^2s/(12\ell)\rfloor\ge2$, and colour the vertices of
$H$ independently and uniformly with $t$ colours. For a fixed vertex
$v$ and colour $i$, let $X$ be the number of out-neighbours of $v$ of
colour $i$. Then $X$ is binomialy distributed with expected value
$\mu=d_H^+(v)/t\ge12\ell/\tau^2$, so the Chernoff bound  gives
$$
\mathbb P\bigl[|X-\mu|>\tau\mu\bigr]\le2e^{-\tau^2\mu/3}\le2e^{-4\ell}.
$$
There are at most $nt\le n^2$ pairs $(v,i)$, since
$t\le s\le|V(H)|-1<n$. Hence, by the union bound the probability that some bound fails is at
most $2n^2e^{-4\ell}=1/(8n^2)<1$. We may therefore fix a colouring in
which every vertex has between $(1-\tau)d_H^+(v)/t$ and
$(1+\tau)d_H^+(v)/t$ out-neighbours in each colour class
$V_1,\ldots,V_t$. In particular, every colour class is non-empty and
$\delta^+(H[V_i])\ge(1-\tau)s/t$ for all $i\in [t]$.

By Theorem~\ref{thm:fractional}, each $H[V_i]$ has a fractional cycle
packing $\{x_C\}$ with total weight
$$
\nu_i^\ast=\sum_{C\subseteq V_i}x_C\ge\frac{(1-\tau)s}{2t}.
$$

Choose $I$ uniformly at random from $\{1,\ldots,t\}$. Conditional on
$I=i$, select a cycle $C$ in $H[V_i]$ with probability $x_C/\nu_i^\ast$.
Let $C^\ast$ denote the cycle thus selected. For a cycle
$C\subseteq V_i$, the probability that $C^\ast=C$ is $x_C/(t\nu_i^\ast)$. Consider an arbitrary vertex $u$ and let $j\in [t]$ be the index so that $u\in V_j$. If $I\ne j$, then $u\notin V(C^\ast)$; if $I=j$, then
$u\in V(C^\ast)$ exactly when the selected cycle in $H[V_j]$ contains
$u$. Hence
$$
\mathbb P[u\in V(C^\ast)]
=\frac{1}{t\nu_i^\ast}\sum_{C\ni u}x_C
\le\frac{1}{t\nu_i^\ast}
\le\frac{2}{(1-\tau)s},
$$
where the inequality uses the capacity constraint
$\sum_{C\ni u}x_C\le1$ and $\nu_i^\ast\ge(1-\tau)s/(2t)$. Summing over
$u\in N_H^+(v)$,
$$
\mathbb E\bigl[|N_H^+(v)\cap V(C^\ast)|\bigr]
=\sum_{u\in N_H^+(v)}\mathbb P[u\in V(C^\ast)]
\le\frac{2d_H^+(v)}{(1-\tau)s}.
$$

Finally, $C^\ast$ lies in a single colour class $V_i$, and therefore
$$
|N_H^+(v)\cap V(C^\ast)|
\le |N_H^+(v)\cap V_i|
\le\frac{(1+\tau)d_H^+(v)}{t}
\le\frac{24(1+\tau)\ell\,d_H^+(v)}{\tau^2s},
$$
where the last inequality uses $t=\lfloor\tau^2s/(12\ell)\rfloor\ge\tau^2s/(24\ell)$.
\end{proof}

The next lemma shows that, after deleting one cycle, the $q$-th power of
the reciprocal out-degree of every vertex does not increase in
expectation. This will be applied at each step of the deletion sequence
in Lemma~\ref{lem:sequence}.

\begin{lemma}\label{lem:deletion}
Let $q\ge1$ be an integer, let $0<a<b<s$, and suppose that
$\delta^+(H)\ge s$. Let $C^\ast$ be a random cycle such that, for every
vertex $v$,
$$
\mathbb E\bigl[|N_H^+(v)\cap V(C^\ast)|\bigr]\le\frac{a\,d_H^+(v)}{s} \mbox{ and } 
|N_H^+(v)\cap V(C^\ast)|\le\rho\,d_H^+(v).
$$  If $(q+1)\rho\le1-a/b$, then $H-V(C^\ast)$ is non-empty and
every vertex outside $C^\ast$ has positive out-degree in
$H-V(C^\ast)$. Moreover, for every vertex $v$,
$$
\mathbb E\left[
\left(\frac{s-b}{d_{H-V(C^\ast)}^+(v)}\right)^q
\right]
\le\left(\frac{s}{d_H^+(v)}\right)^q.
$$
\end{lemma}

\begin{proof}
The hypothesis gives
$$
\frac{|N_H^+(v)\cap V(C^\ast)|}{d_H^+(v)}\le\rho<\frac1{q+1}<1,
$$
so $d_{H-V(C^\ast)}^+(v)>0$ for every $v\notin V(C^\ast)$, and
$H-V(C^\ast)$ is nonempty.

For $0\le x<1/(q+1)$, Lemma~\ref{lem:mvt} gives
$(1-x)^{-q}\le1+qx/(1-(q+1)x)$. Substitute $x=|N_H^+(v)\cap V(C^\ast)|/d_H^+(v)\le\rho\leq 1/(q+1)$. Using
$1-(q+1)\rho\ge a/b$,
$$
\left(\frac{d_H^+(v)}{d_H^+(v)-|N_H^+(v)\cap V(C^\ast)|}\right)^q
\le1+\frac{qb}{a\,d_H^+(v)}|N_H^+(v)\cap V(C^\ast)|.
$$
Taking expectations over the choice of $C^\ast$ and using
$\mathbb E[|N_H^+(v)\cap V(C^\ast)|]\le a\,d_H^+(v)/s$,
$$
\mathbb E\left[
\left(\frac{d_H^+(v)}{d_H^+(v)-|N_H^+(v)\cap V(C^\ast)|}\right)^q
\right]
\le1+\frac{qb}{s}\le e^{qb/s}.
$$

Since $(1-b/s)^q\le e^{-qb/s}$, multiplying both sides by
$(1-b/s)^q(s/d_H^+(v))^q$ and using that $d_{H-V(C^\ast)}^+(v)=d_H^+(v)-|N_H^+(v)\cap V(C^\ast)|$ yields  
$$
\mathbb E\left[
\left(\frac{s-b}{d_{H-V(C^\ast)}^+(v)}\right)^q
\right]
\le\left(\frac{s}{d_H^+(v)}\right)^q,
$$
which completes the proof.
\end{proof}

We now prescribe a decreasing sequence $\{s_i\}$ of out-degree bounds.
The decrease at each step approaches two when the current bound
is large. A probability estimate for each vertex shows that there
is a sequence of cycle deletions satisfying all the bounds.

\begin{lemma}\label{lem:sequence}
Let $n\ge2$ be an integer, and let D be a digraph with
$|V(D)|\le n$ and minimum out-degree $r\ge1$.
Put $\ell=\log(2n)$.
Fix an integer $q>\ell$, and set
$s_0=r(1-\ell/q)$, $\kappa=30(q+1)\ell$, and $s_\ast=64\kappa$.
If $s_0\ge s_\ast$, then D contains a family of $m$
disjoint cycles with
$m\ge s_0/2-(3/2)\kappa^{1/3}s_0^{2/3}-8\kappa$.
Their deletion leaves a non-empty induced subdigraph with minimum
out-degree greater than $s_\ast-4$.
\end{lemma}

\begin{proof}
Set $\tau_j=(\kappa/s_j)^{1/3}$ and $b_j=2/(1-2\tau_j)$.
Starting from $s_0$, set $s_{j+1}=s_j-b_j$ until the first $m$ with
$s_m<s_\ast$. For every $j<m$, $s_j\ge s_\ast=64\kappa$, and hence
$\tau_j\le1/4$ and $2\le b_j\le4$. Thus $m$ is finite. The final digraph
has minimum out-degree greater than $s_m\ge s_\ast-4$.

At step $j$, the current digraph $H_j$ has minimum out-degree greater than
$s_j$. Since $s_j\ge64\kappa$, we have $s_j^{1/3}\ge4\kappa^{1/3}$ and
thus
$$
24\tau_j^{-2}\ell
=\frac{24\ell\,s_j}{s_j^{1/3}\kappa^{2/3}}
\le\frac{6\ell}{\kappa}s_j\le s_j,
$$
so Lemma~\ref{lem:distribution} with parameter $\tau_j$ gives a probability distribution on the cycles of $H_j$. It satisfies
$$
\mathbb E[|N_{H_j}^+(v)\cap V(C^\ast)|]
\le\frac{2d_{H_j}^+(v)}{(1-\tau_j)s_j},
\qquad
|N_{H_j}^+(v)\cap V(C^\ast)|
\le\frac{24(1+\tau_j)\ell\,d_{H_j}^+(v)}{\tau_j^2s_j}.
$$
These are the hypotheses of Lemma~\ref{lem:deletion} with
$a=2/(1-\tau_j)$, $b=b_j$, and
$\rho=24(1+\tau_j)\ell/(\tau_j^2s_j)$. Note that
$0<a<b_j\le4<s_j$. Moreover,
$$
(q+1)\rho =\frac{24(1+\tau_j)(q+1)\ell}{\tau_j^2s_j}\le\frac{30(q+1)\ell}{\tau_j^2s_j}=\tau_j
<1-\frac{a}{b_j}.
$$
Thus Lemma~\ref{lem:deletion} applies at every step.
%\jbjcomment{perhaps the last in-equality should be justified by a calculation (I checked that it holds)}\\

We now choose $m$ successive cycles at random. Start with $H_0=D$. At
step $j<m$, choose a cycle $C_{j+1}$ from the distribution above and set
$H_{j+1}=H_j-V(C_{j+1})$. If any remaining vertex has out-degree at most
$s_{j+1}$ in $H_{j+1}$, stop; otherwise continue.
For each vertex $v$, define the random variable $Z_j(v)$ is follows 
$$
Z_j(v)=
\begin{cases}
0,&v\text{ has been removed},\\
(s_j/d_{H_j}^+(v))^q,&\text{otherwise}.
\end{cases}
$$
Lemma~\ref{lem:deletion} gives $\mathbb E[Z_{j+1}(v)]\le Z_j(v)$ at
every step. Iterating,
$$
\mathbb E[Z_m(v)]\le Z_0(v)=\left(\frac{s_0}{r}\right)^q
=\left(1-\frac{\ell}{q}\right)^q<1.
$$

If the process stops early, some remaining vertex $v$ has out-degree at
most the current threshold, so $Z_m(v)\ge1$.  This holds because $d_{H_m}^+(v)\le s_m$, hence $s_m/d_{H_m}^+(v)\ge1$.  Since $Z_m(v)\ge0$ and
$\mathbb E[Z_m(v)]\le(1-\ell/q)^q$, Markov's inequality gives
$$
\mathbb P[Z_m(v)\ge1]\le\mathbb E[Z_m(v)]\le\left(1-\frac{\ell}{q}\right)^q.
$$
Taking a union bound over all vertices, the probability of stopping early
is at most
$$
n\left(1-\frac{\ell}{q}\right)^q\le ne^{-\ell}=\frac12.
$$
Therefore, with positive probability all $m$ deletions succeed. Choosing
such a sequence of cycles gives $m$ disjoint cycles, and the remaining
digraph has minimum out-degree greater than $s_m\ge s_\ast-4$.

It remains to estimate $m$. Since $s_j-s_{j+1}=b_j$,
$$
m=\sum_{j=0}^{m-1}1=\sum_{j=0}^{m-1}\frac{s_j-s_{j+1}}{b_j}.
$$
Note that $1/b_j=1/2-(\kappa/s_j)^{1/3}$, and this function of $s_j$ is
increasing. %On the interval 
For all $u\in [s_{j+1},s_j]$ we have 
$$
\frac12-\left(\frac{\kappa}{u}\right)^{1/3}\le\frac1{b_j}.
$$
Therefore
$$
\frac{s_j-s_{j+1}}{b_j}
\ge\int_{s_{j+1}}^{s_j}\left(\frac12-\left(\frac{\kappa}{u}\right)^{1/3}\right)du.
$$
Summing over $j$, the intervals cover $[s_\ast,s_0]$, and
$$
m\ge\int_{s_\ast}^{s_0}\left(\frac12-\left(\frac{\kappa}{u}\right)^{1/3}\right)du
=\frac{s_0}{2}-\frac32\kappa^{1/3}s_0^{2/3}-8\kappa,
$$
using $s_\ast=64\kappa$.
\end{proof}

\begin{theorem}\label{thm:local}
Every digraph $D$ of order $n$ and minimum out-degree $r\ge1$
satisfies
$\nu(D)\ge\max\{0,r/2-7r^{3/4}\sqrt{\log(2n)}\}$.
\end{theorem}

\begin{proof}
Put $\ell=\log(2n)$ and choose $q=\lceil r^{1/4}\sqrt\ell\rceil$.
If $r/2<7r^{3/4}\sqrt\ell$, then $\nu(D)\ge0$ already gives the result.
We may therefore assume $r\ge14r^{3/4}\sqrt\ell$. This implies that $r^{1/4}\geq 14\sqrt{\ell}$ so $q>\ell$ holds.

Set $s_0=r(1-\ell/q)$ and $\kappa=30(q+1)\ell$. Then
$s_0\ge r-r^{3/4}\sqrt\ell>13r/14$.
Since $q\le r^{1/4}\sqrt{\ell}+1$, we have
$\kappa=30(q+1)\ell \le 30(r^{1/4}\sqrt{\ell}+2)\ell<35r^{1/4}\ell^{3/2}$.
Our assumption $r\ge14r^{3/4}\sqrt{\ell}$ ensures $64\kappa<s_0$, because $s_0\ge 13r/14$ and $64\kappa=1920(q+1)\ell$ is dominated by $r$ under our parameter bound.
Thus $s_\ast=64\kappa<s_0$, so Lemma~\ref{lem:sequence} applies and
gives
$$
\nu(D)\ge\frac{s_0}{2}-\frac32\kappa^{1/3}s_0^{2/3}-8\kappa.
$$
The two error terms are at most   $5r^{3/4}\sqrt\ell$ and 
$\frac{10}{7}r^{3/4}\sqrt\ell$, 
respectively. Hence
$
\nu(D)\ge\frac r2-7r^{3/4}\sqrt\ell$.
\end{proof}

% \begin{proof}
% Put $\ell=\log(2n)$ and choose $q=\lceil r^{1/4}\sqrt\ell\rceil$.
% If $r/2<7r^{3/4}\sqrt\ell$, then $\nu(D)\ge0$ already gives the result.
% We may therefore assume $r\ge14r^{3/4}\sqrt\ell$. This implies that $r^{1/4}\geq 14\sqrt{\ell}$ so $q>\ell$ holds.

% Set $s_0=r(1-\ell/q)$ and $\kappa=30(q+1)\ell$. Then
% $s_0\ge r-r^{3/4}\sqrt\ell>13r/14$, and
% $\kappa<35r^{1/4}\ell^{3/2}$. \jbjcomment{Please make the calculation for $\kappa$} The assumption $r\ge14r^{3/4}\sqrt\ell$
% implies $s_\ast=64\kappa<s_0$ \jbjcomment{Also this calculation}, so Lemma~\ref{lem:sequence} applies and
% gives
% $$
% \nu(D)\ge\frac{s_0}{2}-\frac32\kappa^{1/3}s_0^{2/3}-8\kappa.
% $$
% The two error terms are at most   $5r^{3/4}\sqrt\ell$ and 
% $\frac{10}{7}r^{3/4}\sqrt\ell$, 
% respectively. Hence
% $
% \nu(D)\ge\frac r2-7r^{3/4}\sqrt\ell$.
% \end{proof}

% We bound the three terms:
% $$
% \frac{s_0}{2}\ge\frac r2-\frac12r^{3/4}\sqrt\ell,
% $$
% $$
% \frac32\kappa^{1/3}s_0^{2/3}
% <\frac32(35r^{1/4}\ell^{3/2})^{1/3}r^{2/3}
% <5r^{3/4}\sqrt\ell,
% \quad\text{and}\quad
% 8\kappa<280r^{1/4}\ell^{3/2}<\frac{10}{7}r^{3/4}\sqrt\ell,
% $$
% where the last inequality uses $r^{1/2}\ge196\ell$.

\section{Proof of the main theorem}\label{sec:mainproof}

We are now ready to prove the main theorem which we restate below.

\MainThm
\begin{proof}
We establish the stronger explicit bound
$$
f(k)\le\left\lceil2k+64k^{3/4}\sqrt{\log(4k)}\right\rceil
\quad\text{for }k\ge10^{12}.
$$
Since $\log(4k)\le2\log k$ for $k\ge4$, this implies
Theorem~\ref{thm:main} with $C=64\sqrt2$ and $K=10^{12}$.

Set $E_k=k^{3/4}\sqrt{\log(4k)}$ and $r=\lceil2k+64E_k\rceil$.
The function $x^{-1/4}\sqrt{\log(4x)}$ is decreasing for $x\ge10^{12}$, so
$$
\frac{64E_k}{k}\le64\cdot10^{-3}\sqrt{\log(4\cdot10^{12})}<\frac{48}{125}<1,
$$
and hence $2k\le r<3k$.

Suppose there is a counterexample at $(r,k)$,
i.e.\ a digraph $D$ with $\delta^+(D)\ge r$ and $\nu(D)<k$. By
Lemma~\ref{lem:structure}, take one of minimum order $n$ that is
$r$-out-regular. Proposition~\ref{prop:order} gives $n<9k^3$, so
$\log(2n)<3\log(4k)$.

Theorem~\ref{thm:local} yields
$$
\nu(D)\ge\frac r2-7r^{3/4}\sqrt{\log(2n)}>k,
$$
where the strict inequality follows because $r\ge2k+64E_k$, while
$r<3k$ and $\log(2n)<3\log(4k)$ give
$7r^{3/4}\sqrt{\log(2n)}\le7\cdot3^{5/4}E_k<32E_k$, a contradiction.

This shows that no counterexample at $(r,k)$ exists. Hence every
digraph with minimum out-degree at least $r$ contains $k$ disjoint
cycles, so $f(k)\le r=\lceil2k+64E_k\rceil$, as claimed.
\end{proof}

\section*{Use of AI-assisted tools} 

Following Lemma~\ref{lem:deletion}, we establish the $o(k)$ bound using martingale methods. We further attempted to obtain an explicit linear bound of the form $2k+C$ for some absolute constant $C$, but this attempt was unsuccessful. Nevertheless, AI-assisted refinement allowed us to improve the asymptotic error term from $2k+o(k)$ to the explicit bound $O(k^{3/4}\sqrt{\log k})$. Additionally, AI tools aided the construction of the sequence $Z_j(v)$ in Lemma~\ref{lem:sequence}: specifically, removed vertices are assigned a value of zero, and all values are frozen after the first unsuccessful deletion operation.

\end{document}